\documentclass{article}
\usepackage{CJK,CJKnumb,CJKulem,times,dsfont,ifthen,mathrsfs,latexsym,amsfonts,color}
\usepackage{amsmath,amsthm,makeidx,fontenc,amssymb,bm,graphicx,psfrag,listings,curves,extarrows,enumitem}
\usepackage[
linktocpage=true,colorlinks,citecolor=magenta,linkcolor=blue,urlcolor=magenta]{hyperref}
 
\usepackage{cite}
\usepackage{geometry}
\usepackage{indentfirst}

\usepackage{marvosym}

\usepackage{amssymb}
\makeatletter

\newcommand{\Rmnum}[1]{\expandafter\@slowromancap\romannumeral #1@}

\makeatother

\numberwithin{Assumption}{section} \numberwithin{Corollary}{section}
\numberwithin{Definition}{section} \numberwithin{equation}{section}
\numberwithin{Example}{section} \numberwithin{Lemma}{section}
\numberwithin{Proposition}{section} \numberwithin{Remark}{section}
\numberwithin{Theorem}{section}

\newtheorem{definition}{Definition}[section]
\newtheorem{theorem}{Theorem}[section]
\newtheorem{lemma}{Lemma}[section]
\newtheorem{corollary}{Corollary}[section]

\newtheorem{remark}{Remark}[section]

\newcommand{\bt}{\begin{theorem}}
	\newcommand{\et}{\end{theorem}}
\newcommand{\bl}{\begin{lemma}}
	\newcommand{\el}{\end{lemma}}
\newcommand{\bd}{\begin{definition}}
	\newcommand{\ed}{\end{definition}}
\newcommand{\bc}{\begin{corollary}}
	\newcommand{\ec}{\end{corollary}}
\newcommand{\bp}{\begin{proof}}
	\newcommand{\ep}{\end{proof}}
\newcommand{\bx}{\begin{example}}
	\newcommand{\ex}{\end{example}}
\newcommand{\bi}{\begin{exercise}}
	\newcommand{\ei}{\end{exercise}}
\newcommand{\br}{\begin{remark}}
	\newcommand{\er}{\end{remark}}
\newcommand{\be}{\begin{equation}}
	\newcommand{\ee}{\end{equation}}
\newcommand{\bal}{\begin{align}}
	\newcommand{\bn}{\begin{enumerate}}
		\newcommand{\en}{\end{enumerate}}
	\newcommand{\ba}{\begin{align}}
		\newcommand{\ea}{\begin{align}}
			\newcommand{\bg}{\begin{align*}}
				\newcommand{\eg}{\end{align*}}
			\newcommand{\bcs}{\begin{cases}}
				\newcommand{\ecs}{\end{cases}}

			\newcommand{\N}{{\mathbb N}}

			\newcommand{\R}{{\mathbb R}}
			\newcommand{\bean}{\begin{eqnarray*}}
				\newcommand{\eean}{\end{eqnarray*}}

			\def\text#1{{\rm #1}}

			\def\dis{\displaystyle}

			\setitemize{itemindent=38pt,leftmargin=0pt,itemsep=-0.4ex,listparindent=26pt,partopsep=0pt,parsep=0.5ex,topsep=-0.25ex}

			\numberwithin{equation}{section}
			
\begin{document}
				\theoremstyle{plain}

			\title{\bf  Normalized solutions for the nonlinear Schr\"{o}dinger equation with  potentials \thanks{ 	
					E-mails: pxq52918@163.com (X.  Peng),   mrizzi1988@gmail.com, matteo.rizzi@uniba.it (M. Rizzi)}}
			
			\date{}
			\author{
				{\bf    Xueqin Peng$^{\mathrm{a}}$,\; Matteo Rizzi$^{\mathrm{b}}$ \textsuperscript{\Letter}}\\
				\footnotesize \it  $^{\mathrm{a}}$ Department of Mathematical Sciences, Tsinghua University, Beijing 100084, China\\
				\footnotesize \it	$^{\mathrm{b}}$ Dipartimento di Matematica, Universit\'{a} degli studi Aldo Moro di Bari, Bari 70125, Italy
			}

				\maketitle
				\begin{center}
					\begin{minipage}{130mm}
						\begin{center}{\bf Abstract }\end{center}
						
					In this paper, we find normalized solutions to the following Schr\"{o}dinger equation 
					\begin{equation*}
						\begin{cases}
							-\Delta u-\frac{\mu}{|x|^2}h(x)u+\lambda u =f(u)\quad\text{in}\quad\mathbb{R}^{N},\\
							\dis	u>0,\quad \int_{\mathbb{R}^{N}}u^2dx=a^2,
						\end{cases}
					\end{equation*}
					where  $N\geq3$, $a>0$ is fixed, $f$ satisfies mass-subcritical growth conditions and $h$ is a given bounded function with $||h||_\infty\le 1$. The $L^2(\R^N)$-norm of $u$ is fixed and $\lambda$ appears as a Lagrange multiplier. Our solutions are constructed by minimizing the corresponding energy functional on a suitable constraint. Due to the presence of a possibly nonradial term $h$, establishing compactness becomes challenging. To address this difficulty, we employ the splitting lemma to exclude both the vanishing and the dichotomy of a given any minimizing sequence for appropriate \( a > 0 \).
					Furthermore, we show that if $h$ is radial, then radial solutions can be obtained for any $a>0$. In this case, the radial symmetry allows us to prove that such solutions converge to a ground state solution of the limit problem as $\mu \to 0^+$.


						%
						%


						
						\vskip0.23in
						
					\end{minipage}
				\end{center}

					\vskip0.23in
					\section{Introduction}
					\setcounter{equation}{0}
					\setcounter{Assumption}{0} \setcounter{Theorem}{0}
					\setcounter{Proposition}{0} \setcounter{Corollary}{0}
					\setcounter{Lemma}{0}\setcounter{Remark}{0}
					\par
					Consider the following time-dependent Schr\"{o}dinger  equation
					\begin{equation}\label{eq1.1}
						i\partial_t\psi+\Delta\psi+V(x)\psi +g(|\psi|)\psi=0\qquad\text{in}~(0,\infty)\times\mathbb{R}^{N
},
					\end{equation}
					where the function $\psi=\psi(t,x):(0,\infty)\times \mathbb{R}^N\rightarrow\mathbb{C}$ is complex valued, $g$ is real valued and $V:\R^N\to\R$ is a given potential. An important topic is to search standing wave solutions of \eqref{eq1.1}, namely solutions of the form
\begin{equation*}
    \psi(t,x)=e^{i\lambda t}u(x),\quad\lambda\in\R,\quad u:\R^N\to \mathbb{R}.
\end{equation*}
In this case, $\psi$ is a solution to \eqref{eq1.1} if and only if $u$ solves the following semilinear elliptic equation	\begin{equation}\label{eq1.2}
-\Delta u+\lambda u+V(x)u=f(u)\quad\text{in}\quad \R^N,
\end{equation}
			where $f(u)=g(|u|)u$.
				Because of its importance in many different physical frameworks, many authors have investigated the
Schr\"{o}dinger equation. At this point, there are two different ways to approach \eqref{eq1.2} according to the role of $\lambda$:
\begin{description}
    \item[(1)] the frequency $\lambda$ is a fixed and assigned parameter,
     \item[(2)] the $L^2(\R^N)$-norm of the solution is given and the frequency $\lambda$ is an unknown of the problem and appears as a Lagrange multiplier.
\end{description}
In the first case, there are many results in the literature which are focused on existence, nonexistence and multiplicity of solutions to (\ref{eq1.2}). Some of them can be found, for instance, in 
\cite{BL1983,LPZ,L1984} and the references therein. In those papers the authors apply variational methods (see e.g. \cite{Tru, Rab1, SM, m1}) to construct such solutions. In fact, it is easy to see that the critical points of the following functional 
\begin{equation*} 
							I (u)=\frac{1}{2}\int_{\mathbb{R}^N}|\nabla u|^2dx+\frac{1}{2}\int_{\mathbb{R}^N}(V(x)+\lambda)u^2dx
							- \int_{\mathbb{R}^N}F(u)dx
						\end{equation*}
give rise to solutions of (\ref{eq1.2}).
 
					\vskip 0.1in
					On the other hand, the second case has   attracted significant attention during the last few years. It is normally referred to as the \textit{fixed mass problem}, which consists in looking for solutions to \eqref{eq1.2} on the $L^2$-constraint
					\begin{equation}\label{eq1.3}
						 \mathcal{S}_a:=\{u\in H^1(\R^N):||u||_2=a\}.
					\end{equation}
		 These kind of solutions, which have recently gained interest in physics, are known as \textit{normalized solutions}. Indeed, from the physical point of view, this approach turns out to be more meaningful since it offers a better insight of the variational properties of the stationary solutions of (\ref{eq1.2}), such as stability or instability. Another physical motivation to study normalized solutions is that the $L^2(\R^N)$-norm of the standing wave solutions to \eqref{eq1.1} is preserved in time.\\ 
         
         Compared to the fixed frequency case, the study of normalized solutions is more complicated because the terms in the corresponding energy functional scale differently. The main difficulties and challenges are based on several aspects: (i)  the classical mountain-pass theorem is not applicable to construct a  $(PS)$ sequence; (ii) since $\lambda$ is unknown, the Nehari manifold cannot be used anymore, which brings more obstacles to prove the boundedness of the $(PS)$ sequence; (iii)  the compactness of the $(PS)$ sequence is more challenging even if one considers the radial case since the embedding of $H^1_r(\R^N)\hookrightarrow L^2(\R^N)$ is not compact.

						\vskip 0.1in
		Most of papers are devoted to the autonomous case, that is $V(x)\equiv0$. For results in this direction, we refer to e.g. \cite{B2013,JJ,PPV,S1,S2,WW}. However, the case $V(x)\not\equiv 0$ is totally different, especially if $V$ is not radial, so that it is not possible to work in the radial setting and use the compactness of the embedding $H_r^1(\R^N)\hookrightarrow L^p(\R^N)$, for $p\in(2,2^*)$. As a consequence, it is much more complicated to recover the compactness of (PS)-sequences. From the perspective of physics, $V(x)$ represents an external potential that influences the behavior of the stationary waves.\\
        
        A natural approach to such a question is by variational methods.
		Indeed, if  $V \in  L^r(\R^N )$ for some $r \in [\frac{N}{2}, +\infty], r \geq 1$, solutions to \eqref{eq1.2}-\eqref{eq1.3} 
		can be found as critical points of the energy functional
			\begin{equation} \label{def-E}
		E(u)=\frac{1}{2}\int_{\mathbb{R}^N}|\nabla u|^2dx+\frac{1}{2}\int_{\mathbb{R}^N}V(x)u^2dx
		- \int_{\mathbb{R}^N}F(u)dx
		\end{equation}
		on $\mathcal{S}_a$. 
		Now we recall some existing results in this direction. It turns out that the geometry of $E$
		on $\mathcal{S}_a$ changes according to the growth condition of $f$ in relation with the mass critical exponent $\bar{p}=2+\frac{4}{N}$. In this respect, among many results, we refer the reader to the
		recent paper \cite{Ikoma=CV} by Ikoma and Miyamoto, where the existence of global minimizers is obtained for non-positive potentials. After
					that, Ikoma and Miyamoto  \cite{IM2} applied their argument to a minimization problem
					with two constraint conditions and potentials. On the other hand, in the mass-supercritical case, the functional $E$ is unbounded
					from below on $\mathcal{S}_a$, thus the problem cannot be solved by (global) minimization arguments, and
					even the appropriate definition of ground state is not clear. For  (step well) trapping potentials, i.e. potentials 
					satisfying
					\begin{equation*}
						 \lim_{|x|\rightarrow +\infty}V(x)=+\infty,
					\end{equation*}
					the trapping nature of the potential provides
				compactness, which guarantees the existence of solutions which are local minimizers of $E$ on $\mathcal{S}_a$, see e.g. \cite{NTV}.  In  \cite{bm}, Bartsch et al. considered weakly repulsive potentials,
					i.e.
					\begin{equation*}
						  V(x)\geq \liminf_{|x|\rightarrow +\infty}V(x)>-\infty 
					\end{equation*}
					and they applied a new variational principle exploiting the Pohozaev identity which can be used to obtain existence of solutions with high Morse index. For weakly attractive potential, that is
						\begin{equation*}
						V(x)\leq \limsup_{|x|\rightarrow +\infty}V(x)<+\infty ,
					\end{equation*}
				Molle et al. \cite{MRV} obtained a mountain pass solution and a local minimizer, under suitable assumptions about $f$ and $V$. For other contributions in this direction, we refer the reader to \cite{LM, PR, ZZ} and the references therein.

							\vskip0.1in
				It is interesting to mention that, up to now, it has not been considered that problem  \eqref{eq1.2}-\eqref{eq1.3} in the $L^2$-subcritical regime involving  the ``Hardy potential" (or ``inverse-square potential")  $V(x):=-\frac{\mu}{|x|^2}$ in the linear part except \cite{LZ}, where they studied the mixed non-linearities by using the Pohozaev manifold approach developed by Soave \cite{S1,S2}. The potential with this decay rate plays a critical role in non-relativistic quantum mechanics, as it represents an intermediate threshold between two distinct classes of potentials: regular potential and singular potential, we refer to  \cite{FLS} for more details. Besides, it also arises in many other
						areas such as quantum mechanics, nuclear physics, molecular physics and quantum cosmology \cite{Fe}.
						\vskip 0.1in
						 Due to its broad relevance,  one of the interest of the present paper is to investigate this problem.  Precisely, we consider the following problem:
							\begin{equation}\label{q}
								\begin{cases}
									-\Delta u-\frac{\mu}{|x|^2} h(x)\color{black}u+\lambda u=f(u)\quad\text{in}\quad\mathbb{R}^{N} ,\tag{$\mathcal{Q}$}\\
									 	u>0,~ \int_{\mathbb{R}^{N}}u^2dx=a^2,
								\end{cases}
							\end{equation}
							where $0\le\mu<\bar{\mu}:=\left(\frac{N-2}{2}\right)^2$, $h\in L^\infty(\R^N)$ with $\|h\|_\infty\le 1,\, h\ne 0$ and $f$ satisfies the following conditions:
							
							\vskip0.1in
							\noindent	$(f_1)$ 					 $f(t)\in C(\R,\R)$ is odd   and there exist constants $q   \in (2,2+\frac{4}{N})$ and $\alpha\in(0,\infty)$ \color{black}such that 
							\begin{equation}
								\lim\limits_{ t\rightarrow 0^+}\color{black}\frac{f(t)}{t^{q-1}}=\alpha.
							\end{equation}
						 
							\noindent	$(f_2)$ There is $r\in(2,2+\frac{4}{N})$ \color{black} such that $\frac{f(t)}{t^{r-1}}$ is increasing on $(0,+\infty)$.
							\vskip0.1in
								\noindent	$(f_3)$ There exist $c_0,t_0>0$ and $p\in(2,2+\frac{4}{N})$ such that 
								 \begin{equation*}
								f(t)\le c_0 t^{p-1},\quad\forall\,t>t_0.
                                \end{equation*} 
							\vskip0.05in 
						 	\color{black}The easiest example of non-linearity $f$ satisfying $(f_1)-(f_3)$ is $f(t)=|t|^{p-2}t$ with $p\in(2,2+\frac{4}{N})$. However, it is possible to verify that the function 
						 	$
						 		f(t)=|t|^{q-2}t +|t|^{p-2}t\ln(1+|t|)
						 $
						  satisfies $(f_1)-(f_3)$ as well, provided $p,\,q\in(2,2+\frac{4}{N})$. The non-linearities which fulfill $(f_1)-(f_3)$ will be referred to as \textit{mass-subcritical} non-linearities.\\

                          Problem \eqref{q} has a variational structure. In fact, under the assumptions $(f_1)-(f_3)$, the solutions to \eqref{q} are the crtical points of the energy functional
                          \begin{equation}\notag
                          J_\mu(u):=\frac{1}{2}\int_{\R^N}|\nabla u|^2dx-\frac{\mu}{2}\int_{\R^N}\frac{h(x)}{|x|^2}u^2(x)dx-\int_{\R^N}F(u) dx,\quad\forall\, u\in H^1(\R^N)
                          \end{equation}
					on $\mathcal{S}_a$, where $F(t):=\int_0^t f(s)ds$. We note that $J_\mu$ is well defined in $H^1(\R^N)$ for any $\mu\in[0,\bar{\mu})$, with $\bar{\mu}=\left(\frac{N-2}{2}\right)^2$, thanks to the Hardy-Sobolev inequality	
                    \begin{equation*}
                                        \int_{\R^N}\frac{u^s}{|x|^s}dx\leq  (\frac{s}{N-s})^s\int_{\R^N}|\nabla u|^sdx\qquad\forall\,u\in W^{1,s}(\R^N),\,1<s<N
                                    \end{equation*}
                    applied with $s=2$.\\
                    
                        For $\mu=0$, the existence of solutions to Problem \eqref{q} was treated in \cite{shibata}. More precisely, setting
\begin{equation}\label{def-m-0}
m^0(a):=\inf_{u\in\mathcal{S}_a}J_0(u),\qquad\forall\, a>0,
\end{equation}
Shibata \cite{shibata} proved the following result.
\begin{theorem}\label{th-limit-pb}
Assume that $f$ satisfies $(f_1)-(f_3)$. Then there exists $a_0>0$ such that, for any $a>a_0$, there exists a solution $(v_a,\lambda_a)\in H^1(\R^N)\times (0,\infty)$ to the problem
\begin{equation}
                        \label{limit-pb}
                        \begin{cases}
									-\Delta u+\lambda u=f(u)\quad\text{in}\quad\mathbb{R}^{N} ,\\
								\int_{\R^N}u^2(x)dx=a^2,\,u>0,
					\end{cases}   
                        \end{equation}    
such that $J_0(v_a)=m^0(a)\in(-\infty,0)$.
\end{theorem}  
Problem \eqref{limit-pb}, which corresponds to the case $\mu=0$, will be referred to the \textit{limit problem}. We note that the assumptions about $f$ in \cite{shibata} are slightly weaker. However $(f_1)-(f_3)$ are enough in our context.\\

The solution constructed in Theorem \ref{th-limit-pb} corresponds to a global minimizer of $J_0$ on $\mathcal{S}_a$ with negative energy, which exists thanks to the mass-subcritical growth of $f$. In fact, if for instance $f(t)=|t|^{p-2}t$ with $p$ mass-supercritical, that is $p\in(2+\frac{4}{N},2^*)$, where $2^*=\frac{2N}{N-2}$, the energy functional $J_\mu$ is unbounded from below on $\mathcal{S}_a$ for any $\mu\in[0,\bar{\mu})$, a detailed discussion is given in Remark~\ref{rem2.2}.
 However, in some special cases, such as for instance $\mu=0$, solutions to \eqref{q} are known to exist non-linearities with mass-supercritical growth as well. An example of such solutions were constructed through a mountain-pass argument in \cite{JJ}.\\ 


For $\mu\in(0,\bar{\mu})$, a radial ground state solution $U$ to the problem
\begin{equation}
\label{fixed-fr-pb}
\begin{cases}
-\Delta u-\frac{\mu}{|x|^2}u+ u=f(u)\quad\text{in}\quad\mathbb{R}^{N} ,\\
u\in H^1(\R^N),\,u>0,
\end{cases}   
\end{equation}
was constructed in \cite{LLT=JDE} for a quite wide class of non-linearities $f$, including those functions $f$ which satisfy $(f_1)-(f_3)$ with $p,\,q,\,r\in(2, 2^*)$. We stress that in the aforementioned paper the frequency is fixed and there is no mass constraint. Moreover, the non-linearity is allowed to have mass-superctical growth. \\
                        
                        As a consequence, if we take the pure power non-linearity $f(t)=|t|^{p-2}t$ with $p\in(2, 2^*)\setminus \{2+\frac{4}{N}\}$ and $h=1$, a scaling argument proves the existence of a solution to Problem \eqref{q} with $h=1$ for any $a>0$. In fact, given a solution $U$ to \eqref{fixed-fr-pb}, the function $u_a(x):=\nu^{-\frac{2}{p-2}}U(x/\nu)>0$ is a solution to \eqref{q} with $h=1$ and $\lambda=\nu^{-2}$ if and only if $$\nu=\left(\frac{a}{||U||_2}\right)^{\frac{2(p-2)}{N(p-2)-4}}>0.$$
                        However, for other non-linearities and for non-constant $h$, the scaling argument does not work and the existence of solutions to Problem \eqref{q} is a highly non-trivial issue, which is solved in the next Theorem. To the best of our knowledge, this is the first attempt to address the problem in the setting 
                         $h\neq1$.\\

                        More precisely, setting
                        \begin{equation}
                        \label{def-m-mu}
                        m^\mu(a):=\inf_{u\in\mathcal{S}_a}J_\mu(u)  \qquad\forall\, a>0,\,\mu\in[0,\bar{\mu}), 
                        \end{equation}
                        we have the following statement.
                        \begin{theorem}\label{thm1.1}
								Assume that $N\geq3, (f_1)-(f_3)$ hold, $h\in L^\infty(\R^N)$ with $\|h\|_\infty\le 1$ and $h\ne 0$, $h\ge 0$. \color{black}Then  there exists $a_0>0$ such that, for all $a>a_0$ and for all $\mu\in(0,\bar{\mu})$, \color{black}the problem \eqref{q} admits a solution $(u^\mu,\lambda^\mu)\in H^1(\R^N)\times [-\frac{2 m^\mu(a)}{a^2},\infty)$ with $u^\mu>0$ and $J_\mu(u^\mu)=m^\mu(a)\in(-\infty,0)$. 
							\end{theorem}
                        
                        After that, we take $h$ to be radial and construct radial solutions to Problem \eqref{q} for mass-subcritical non-linearities, thanks to the compactness of the embedding $H^1_r(\R^N) \hookrightarrow L^p(\R^N)$ for any $p\in(2, 2^*)$. More precisely, for $a>0$ and $\mu\in[0,\bar{\mu})$ we set
                        $$\mathcal{S}_{a,r}:=\mathcal{S}_a\cap H^1_r(\R^N),\qquad m^\mu_r(a):=\inf_{u\in \mathcal{S}_{a,r}}J_\mu(u)$$
and we prove the following result.
\begin{theorem}\label{thm-rad}
Assume that $N\geq3,(f_1)-(f_3)$ hold, $h\in L^\infty(\R^N)$ is radial and fulfills $\|h\|_\infty\le 1$, $h\ne 0,\, h\ge 0$. Then for all $a>0$ and for all $\mu\in[0,\bar{\mu})$, the problem \eqref{q} admits a solution $(u^\mu,\lambda^\mu)\in H^1_r(\R^N)\times [-\frac{2 m_r^\mu(a)}{a^2},\infty)$ with $u^\mu>0$ and $J_\mu(u^\mu)=m^\mu_r(a)\in(-\infty,0)$.   Particularly, if $h=1$, then $u^\mu(x)\in C^2(\R^N \setminus \{0\})$ and
there are positive constants $\delta$ and $R$ such that for $|\alpha| \leq 2$,
\begin{equation*}
	|D^{\alpha}u^{\mu}(x)|\leq C\,\text{exp}(-\delta\, |x|),\quad \forall \, |x|\geq R
\end{equation*}
and
\begin{equation*}
	u^\mu(x)\leq C|x|^{-(\sqrt{\bar{\mu}}-\sqrt{\bar{\mu}-\mu})},\quad\, \forall\, x\in B_{\rho}(0)\setminus \{0\},
\end{equation*}
where  $\rho>0$ is sufficiently small.
 	\color{black}
\end{theorem}      
We note that, for $\mu=0$, we include the result by Shibata \cite{shibata} for non-linearities $f$ fulfilling $(f_1)-(f_3)$, while the result is new for $\mu\in(0,\mu_0)$. 
\vskip 0.1in 
Finally, we describe the asymptotic behaviour of the radial solutions to \eqref{q} constructed in Theorem \ref{thm-rad} as $\mu\to 0^+$. We stress that the radial symmetry of $u^\mu$ is crucial here, since we will use the compactness of the embedding $H^1_r(\R^N)\hookrightarrow
 L^p(\R^N)$ for any $p\in(2,2^*)$. In particular, we will see that $u^\mu$ approaches a radial minimizer $u^0$ of $J_0$ on $\mathcal{S}_a$ as $\mu\to 0^+$.
\begin{theorem}\label{th1.2}
Let $a_0>0$ be the constant introduced in the statement of Theorem \ref{thm1.1}. Assume that $ (f_1)-(f_3)$ hold, $\mu\in(0,\bar{\mu})$, $h\in L^\infty(\R^N)$ is radial and fulfills $\|h\|_\infty\le 1$, $h\neq0$, $h\ge 0$. \color{black}  
 Let $(u^{\mu},\lambda^{\mu})\in H^1_r(\R^N)\times [-\frac{2 m^\mu(a)}{a^2},\infty)$ be the solution of \eqref{q} constructed in Theorem \ref{thm-rad} with $a>a_0$. Then there exists a solution $(u^{0},\lambda^{0})\in H^1_r(\R^N)\times [-\frac{2 m^0(a)}{a^2},0)$ to \eqref{q} with $\mu=0$ such that $u^{\mu}\to u^0$ strongly in $H^1(\R^N)$, $\lambda^\mu\to\lambda^0$ as $\mu\to 0^+$ and  $J_0(u^0)=m^0(a)\in(-\infty,0)$.
\end{theorem}
                            

						In the paper we will need the Gagliardo-Nirenberg inequality.	\color{black}\begin{remark}
								We recall the  Gagliardo-Nirenberg inequality \cite{L1959}:  for any $N\geq 3$ and $p\in[2,2^*]$ we have
									\begin{equation}
										\|u\|_p\leq C(N,p)\|\nabla u\|_2^{\mu_p}\|u\|_2^{1-\mu_p},
									\end{equation}
									where $\mu_p=N(\frac{1}{2}-\frac{1}{p})$ and $2^*= \frac{2N}{N-2}$.

							\end{remark}

				\vskip0.1in

				The paper is organized as follows. Section \ref{sec2} is dedicated to the proof of Theorem \ref{thm1.1}. In Section \ref{sect3}, we will treat the radial case and in Section \ref{se2}, we study the asymptotic behavior of the global minimizer. 

				\medspace
				\vskip 0.1in

							\textbf{Notations} 
							Throughout this paper, we make use of the following notations:
							\vskip 0.1in
							\begin{itemize}
								\item $L^p(\mathbb{R}^N)$ ($p\in[1,\infty)$) is the Lebesgue space equipped with the norm
								\begin{equation*}
									\|u\|_p=\left(\int_{\mathbb{R}^N}|u|^pdx\right)^{\frac{1}{p}};
								\end{equation*}
								\item $L^{\infty}(\mathbb{R}^N)$  is the Lebesgue space equipped with the norm
								\begin{equation*}
									\|u\|_{\infty}=\mathop{\text{ess~sup}}\limits_{x\in \mathbb{R}^N}|u(x)|;
								\end{equation*}
								\item $H^1(\mathbb{R}^N)$ denotes the usual Sobolev space endowed with the norm
								\begin{equation*}
									\|u\|_{H^1}=\left(\int_{\mathbb{R}^N}(|\nabla u|^2+|u|^2)dx\right)^{\frac{1}{2}};
								\end{equation*}
                              \item $H^1_r(\R^N)$ denotes the space of radial functions $u\in H^1(\R^N)$;
                              \item $H^{-1}(\mathbb{R}^N)$ denotes the dual space of $H^1(\mathbb{R}^N)$;
                                \vskip0.05in
								\item $D^{1,2}(\mathbb{R}^N)$ is the Banach space given by
								\begin{equation*}
									D^{1,2}(\mathbb{R}^N)=\{u\in L^6(\mathbb{R}^N):\nabla u\in L^2(\mathbb{R}^N)\};
								\end{equation*}
								\item $C$, $C_1,C_2,c_1,c_2\cdots$ represent positive constants whose values may change from line to line;
                                  \vskip0.05in
								\item $``\rightarrow"$ and $``\rightharpoonup"$ denote the strong and weak convergence in the related function spaces respectively;
                                  \vskip0.05in
								\item $o(1)$ denotes the quantity that tends to $0$;
                                  \vskip0.05in
								\item For any $x\in\mathbb{R}^N$ and $r>0$, $B_r(x):=\{y\in\mathbb{R}^N:|y-x|<r\}$.
							\end{itemize}

							\vskip0.23in

                            \section{Proof of Theorem \ref{thm1.1}}\label{sec2}
					In this  section, we always assume $h\neq0,\, h\ge 0$ and construct a solution $(u,\lambda)\in H^1(\R^N)\times (0,\infty)$ to \eqref{q} such that $u$ is a minimizer of the functional \begin{equation} 
						J_{\mu}(u)=\frac{1}{2}\int_{\mathbb{R}^N}|\nabla u|^2dx-\frac{\mu}{2}\int_{\mathbb{R}^N}\frac{h(x)}{|x|^2}u^2dx
						- \int_{\mathbb{R}^N}F(u)dx
					\end{equation}
					on the constraint $\mathcal{S}_a$, which concludes the proof of Theorem \ref{thm1.1}.
					\vskip 0.1in 
The following lemma  is necessary
for this argument.
\begin{lemma}\label{lemma1}
 (Lions \cite[Lemma I.1]{L1984}) Let $\{u_n\}_{n\in\mathbb{N}}$
  be a bounded sequence in
	 $H^1(\R^N )$ which satisfies
	 \begin{equation*}
	 	\sup_{z\in\R^N }\int_{B_1(z)}|u_n|^2dx \rightarrow  0,\quad \text{as}\quad  n\rightarrow \infty.
	 	\end{equation*}
	 	Then for any $r\in (2,2^*)$, 
	 	\begin{equation*}
	 		||u_n||_{L^r(\R^N)} \rightarrow  0,\quad \text{as}\quad  n\rightarrow \infty.
	 	\end{equation*}
\end{lemma}
                    
                             In the next Lemma
                            we will show that for  any $a>0$, $J _{\mu}$ is bounded from below  and coercive \color{black}on $\mathcal{S}_a$. More precisely, 
                            we have the following statement.
							\begin{lemma}\label{lem2.2}
								Assume that $N\geq3$, $(f_1)-(f_3)$ hold, $\mu\in[0,\bar{\mu})$ and $h\in L^\infty(\R^N)$ with $\|h\|_\infty\le 1$. Then $J_\mu$ is coercive on $\mathcal{S}_a$ and $m^\mu(a)\in(-\infty,0)$.
							
							\end{lemma}
							\begin{proof}
								First, we prove $J _{\mu}$ is bounded from below on $\mathcal{S}_a$.
							  By $(f_1)$ and $(f_3)$, we have 
							 \begin{equation}\label{f'}
							 	|f(t)|\leq  C_1 |t|^{q-1}+C_2 |t|^{p-1}.
							 \end{equation}
							 Consequently, applying the Hardy and Gagliardo-Nirenberg inequalities,  we have for any $u\in \mathcal{S}_a$, 
							 \begin{equation}
                             \label{rel-coercivity}
							 	\begin{aligned}
							 	J_{\mu}(u)&=\frac{1}{2}\int_{\mathbb{R}^N}|\nabla u|^2dx-\frac{\mu}{2}\int_{\mathbb{R}^N}\frac{h(x)}{|x|^2} u^2dx
							 	- \int_{\mathbb{R}^N}F(u)dx\\&\geq 
							 	 \frac{ \bar{\mu}-\mu} {2\bar{\mu}}\color{black}\int_{\mathbb{R}^N}|\nabla u|^2dx -C_1  a^{(1-\mu_q)q} \Big( \int_{\R^N}|\nabla u|^2dx\Big)^{\frac{\mu_qq}{2}}-C_2  a^{(1-\mu_p)p} \Big( \int_{\R^N}|\nabla u|^2dx\Big)^{\frac{\mu_pp}{2}},
							 	\end{aligned}
							 \end{equation} 
							 this implies  $J _{\mu}$ is bounded from below  and coercive \color{black}on $\mathcal{S}_a$ since $\mu<\bar{\mu}$ and $\mu_qq,\mu_pp<2$. This shows that $m^\mu(a)\in\R$. 
							 
							 \vskip0.1in
							 Now we prove $m^\mu(a)<0$   for any $a>0$. 
							 \vskip0.1in The assumption $(f_1)$ yields that $\lim \limits_{t\rightarrow 0}\frac{qF(t)}{|t|^q}=\alpha>0$, then there is $\delta>0$ such that
							 \begin{equation*}
							 	\frac{qF(t)}{|t|^q}\geq \frac{\alpha}{2},\quad,\forall t\in[0,\delta].
							 \end{equation*}
							 Given $u\in  \mathcal{S}_a\cap L^{\infty}(\R^N)$ be a nonnegative function, we consider the following $L^2$-scaling
							\begin{equation}\label{l2}
								u_s(x):=s^{\frac{3}{2}}u(sx) ,\quad \forall  s\in\R,
							\end{equation}
							which conserves the $L^2$-norm, i.e., $||u_s||_2=||u||_2$.  Since $\mu>0,$ by direct computation, one gets
							
							\begin{equation*}
								J_{\mu}(u_s)\leq \frac{s^2}{2}\int_{\mathbb{R}^N}|\nabla u|^2dx 
								-s^{-N} \int_{\mathbb{R}^N}F(s^{\frac{N}{2}}u)dx.
							\end{equation*} 
							Choosing $s>0$ small such that $s^{\frac{N}{2}}u(x)\leq \delta$ for any $x\in \R^N$, thus $F(s^{\frac{N}{2}}u)\geq \frac{\alpha}{2q}s^{{\frac{Nq}{2}}}u^q$. As a consequence, it holds
								\begin{equation*}
								J_{\mu}(u_s)\leq \frac{s^2}{2}\int_{\mathbb{R}^N}|\nabla u|^2dx 
								-\frac{\alpha}{2q}s^{{\frac{N(q-2)}{2}}} \int_{\mathbb{R}^N}u^qdx.
							\end{equation*} 
							Since $q \in  (2, 2 + \frac{4}{N})$, increasing $|s|$ if necessary, we derive that $$\frac{s^2}{2}\int_{\mathbb{R}^N}|\nabla u|^2dx 
							-\frac{\alpha}{2q}s^{{\frac{N(q-2)}{2}}} \int_{\mathbb{R}^N}u^qdx=:A_s<0,$$
							this concludes the proof.
							\end{proof}
						
							\vskip0.1in

It follows from Lemma \ref{lem2.2} that any minimizing sequence $\{u_n\}\subset \mathcal{S}_a$ is bounded. As a consequence, the main difficulty to prove Theorem \ref{thm1.1} is to recover compactness of such minimizing sequences; indeed a given minimizing sequence $\{u_n\}$ could run
off to spatial infinity and/or spread uniformly in space. So, even up to translations, two possible
bad scenarios are possible:  \\
\begin{itemize}
    \item (vanishing) $u_n \rightharpoonup 0;$
    \item  (dichotomy) $u_n\rightharpoonup \bar{u}$ and $0<||\bar{u}||_2<a$.
\end{itemize}
~\\
The general strategy in the applications is to prove that any minimizing sequence weakly converges, up to
translation, to a function $\bar{u}$ which is different from zero, excluding the vanishing case. Then
one has to show that $||\bar{u}||_2=a$, which proves that dichotomy does not occur.  
\vskip0.1in
First we give an useful preliminary lemma.
\begin{lemma}\label{lemma2.4}
	Let $\{u_n\}\subset \mathcal{S}_a$ be a minimizing sequence of
	$m^\mu(a)$.
	\begin{description}
		\item[(1)] Let $|u_n|(x) := |u_n(x)|$, then $\{|u_n|(x)\}$ is also a minimizing sequence of $m ^{\mu}(a)$.
		\item[(2)] There exist $\{v_n\}\subset \mathcal{S}_a$ and $\{\lambda_n\}\subset \R$ such that $\{\lambda_n\}$ is bounded and
		\begin{equation*}
			||u_n-v_n||_{H^{1}}\rightarrow 0,\quad J'_{\mu}(v_n)+\lambda_nQ'(v_n)\rightarrow0\quad\text{strongly\quad in}\quad H^{-1}(\R^N),
		\end{equation*}
		where $Q(u) := \frac{1}{2}\int_{\R^N}|u|^2dx$. 
	\end{description}
\end{lemma}
\begin{proof}
	\textbf{(1)} Noting that $||~|u_n|(x)||_2=||u_n||_2=a$, thus $J_{\mu}(|u_n|)\geq m^\mu(a)$. By  \cite[Theorem 6.17]{LL2001}, one has  $||\nabla |u_n|~||_2\leq ||\nabla u_n||_2$, thus $J_{\mu}(|u_n|)\leq J_{\mu}(u_n)$, which yields that $|u_n|$ is also a minimizing sequence of $m^\mu(a)$.
	\vskip0.1in
	
	\textbf{(2)} First we remark that $Q$ is smooth, $Q '(u)u = 2Q(u)$ and $\mathcal{S}_a=Q^{-1}(a^2/2)$. One can easily check $\mathcal{S}_a$ is a closed Hilbert manifold with codimension $1$. The tangent space of $\mathcal{S}_a$ at $u$ and the tangent derivative $(J_{\mu}|_{\mathcal{S}_a})'$ of $E$ at $u$ are given by  
	\begin{equation*}
		T_u \mathcal{S}_a=\{v\in H^1(\R^N)|~\langle \nabla Q(u),v\rangle _{ L^2\color{black}}=0\},
	\end{equation*}
	\begin{equation*}
		(J_{\mu}|_{\mathcal{S}_a})'(u)=J'_{\mu} (u) -\frac{J'_{\mu} (u)\nabla Q(u)}{||\nabla Q(u)||_{2}}Q'(u),
	\end{equation*}
	where $\nabla Q(u)\in H^1(\R^N)$ is the unique element satisfying $\langle \nabla Q(u),v\rangle _{L^2}=     Q'(u)v$ for every $v\in H^1(\R^N)$.
	\vskip0.1in
	Applying Ekeland’s variational principle, there exists a sequence $\{v_n\} \subset \mathcal{S}_a$
	satisfying
	\begin{equation}\label{eq2.16}
    \begin{aligned}
    &J_\mu(v_n)\le J_\mu(u_n),\\
		&||u_n-v_n||_{H^1}\leq \sqrt{\epsilon_n},\\
        &J_{\mu}(v_n)\leq  J_{\mu}(w)+\sqrt{\epsilon_n}||v_n-w||_{H^1}\quad\text{for\quad each\quad }w\in\mathcal{S}_a,
        \end{aligned}
	\end{equation}
	where $\epsilon_n:=  J_{\mu}(u_n) - m_a \geq  0$. By \eqref{eq2.16} 
    we deduce that
	$\{v_n\}$ is also a minimizing sequence and 
	\begin{equation*}
		|| (J_{\mu}|_{\mathcal{S}_a})' (v_n)||_{(T_{u_n} \mathcal{S}_a)^*}:=\sup \{  | (J_{\mu}|_{\mathcal{S}_a})'(v_n)\psi|:~||\psi||_{H^1}=1,\psi\in T_{v_n} \mathcal{S}_a\}\rightarrow0.
	\end{equation*}
	Since $\{v_n\}$ is bounded in $H^1(\R^N)$, $J'_\mu$
	maps bounded sets into bounded sets and $\nabla Q(v_n)=v_n$, 
	setting, for any $n \geq 1$, 
	\begin{equation*}
		\lambda_n:=-\frac{J'_{\mu}(v_n)\nabla Q(v_n)}{||\nabla Q(v_n)||_{2}^2}=-\frac{J'_{\mu}(v_n)v_n}{a^2},
	\end{equation*}
    we can see that $\lambda_n$ is bounded and
	\begin{equation*}
		J'_{\mu}(v_n)+\lambda_nQ'(v_n)\rightarrow0\quad\text{strongly\quad in}\quad H^{-1}(\R^N).
	\end{equation*}
\end{proof}

\color{black}
It is noteworthy to mention that different from the autonomous case or involving periodic potential, the ``vanishing" can be excluded by invariant of variable. However, it is invalid in our case and we have  to explore   new techniques.  Besides,  the  ``dichotomy" is also more complicated due to the limit equation.  For this purpose, we   analyze the properties of the minimizing energy \( m (a) \) as below. 
\vskip0.1in 
Consider the associated  \textit{limit equation}.
\begin{equation} \label{q_infty}
	\begin{cases}
		-\Delta u +\lambda u=f(u)\quad\text{in}\quad\mathbb{R}^{N},\tag{$Q_{\infty}$}\\
		\dis	u>0,\quad \int_{\mathbb{R}^{N}}u^2dx=a^2 
	\end{cases}
\end{equation}
and associated functional is given by
\begin{equation*}
	J_{0}(u)=\frac{1}{2}\int_{\mathbb{R}^N}|\nabla u|^2dx  
	- \int_{\mathbb{R}^N}F(u)dx
\end{equation*}
on $\mathcal{S}_a$. 
Define
\begin{equation} 
	m^{0}(a):=\inf_{u\in\mathcal{S}_a}J_{0}(u).
\end{equation}

			\begin{lemma}\label{lem2.3}
				Assume that $N\geq3,(f_1)-(f_3)$  hold, $\mu\in(0,\bar{\mu})$ and    $h\in L^\infty(\R^N)$ with $\|h\|_\infty\le 1$. Then for any   $a>0$, it holds 
                \vskip0.015in
               \text{(i)} $ m^\mu(a)$ is continuous  in $a$.
                \vskip0.015in
               \color{black}\text{(ii)} for any $ 0<a_1<a_2$, it holds
				\begin{center}
					$ \frac{a_1^2}{a_2^2}m^\mu(a_2)<m^\mu(a_1)<0.$
				\end{center}
                Particularly, $m^{\mu}(a)$ is  decreasing in $a>0$.
              
                 \text{(iii)} for any $0\leq b,c_j<a$ $(j=1,\cdots,k)$ satisfying $a^2=b^2+\sum_{j=1}^{k}c_j^2,$ it holds 
                \begin{equation*}
                	m^\mu(a)\leq m^\mu(b)+\sum_{j=1}^{k} m^{0} (c_j).
                \end{equation*}
			\end{lemma}
			\begin{proof}
            (i) For any $a> 0$, let $a_n > 0$ and $a_n \rightarrow a$. For every $n\in \mathbb{N}$, let $\{u_n\}\subset \mathcal{S}_{a_n}$ such that $ J_{\mu}(u_n)\leq m^\mu(a_n)+\frac{1}{n} \leq \frac{1}{n}.$
Then Lemma \ref{lem2.2} implies that $\{u_n\}$ is bounded in $H^1(\R^N)$, moreover, we have as $n\rightarrow\infty$,
\begin{equation*}
    m^\mu(a)\leq  J_{\mu}(\frac{a}{a_n}u_n)= J_{\mu}(u_n)+o(1)\leq m^\mu(a_n)+o(1).
\end{equation*}
On the other hand, given a minimizing sequence $\{v_n\}\subset \mathcal{S}_a$ for $ J_{\mu}$, we have
\begin{equation*}
     m^\mu(a_n)\leq  J_{\mu}(\frac{a_n}{a}v_n)= J_{\mu}(v_n)+o(1)=  m^\mu(a)+o(1),
\end{equation*}
which jointly with above inequality, gives that $\lim\limits_{n\rightarrow\infty}  m^\mu(a_n)=  m^\mu(a)$.
            \vskip0.1in 
				(ii) Let $t>1$ such that $a_2=ta_1$. It is sufficient to prove
				\begin{equation*}
				m^\mu(a_2)< t^2m^\mu(a_1).
				\end{equation*}
				Note that $ J_{\mu}(u)= J_{\mu}(|u|)$ for any $u\in H^1(\R^N)$. Let $\{u_n\}\subset \mathcal{S}_{a_1}$ be a nonnegative minimizing sequence with respect
				to the $m^\mu(a_1)$, that is, as $n\rightarrow+\infty$,
				\begin{equation*}
					 J_{\mu}(u_n)\rightarrow m^\mu(a_1).
				\end{equation*}
				Moreover, we know $\{u_n\}$ is bounded in $H^1(\R^N)$  uniformly in $a $. Using the Hardy inequality, 
				\begin{equation*}
					0>m^\mu(a_1)+o(1)=J_{\mu}(u_n)\geq 	\frac{\bar{\mu}-\mu}{2\bar{\mu}}\int_{\mathbb{R}^N}|\nabla u_n|^2dx -\int_{\mathbb{R}^N} F(u_n)dx\geq -\int_{\mathbb{R}^N} F(u_n)dx,
				\end{equation*}
				which implies $\int_{\mathbb{R}^N} F(u_n)dx\geq C>0$ for $n\geq n_0$ for some $n_0>0$ large. 
				Setting $v_n = tu_n$, obviously $v_n\in\mathcal{S}_{a_2}$. By $(f_2)$, 
			 $	\frac{F(s)}{s^r}$ is nondecreasing on $(0,+\infty)$, we obtain the inequality
	 \begin{equation*}
	 	 F(ts)\geq t^rF(s),\quad\forall s\geq0,t\geq1.
	 \end{equation*}
		Therefore, above results lead to that for $n\rightarrow \infty$,
		\begin{equation*}
			\begin{aligned}
			m^\mu(a_2)\leq J_{\mu}(v_n)&=t^2J_{\mu}(u_n)+t^2\int_{\R^N}F(u_n)dx-\int_{\R^N}F(tu_n)dx\\&\leq t^2J_{\mu
}(u_n)+(t^2-t^r)\int_{\R^N}F(u_n)dx< t^2J_{\mu}(u_n)=t^2m^\mu(a_1)-C,
			\end{aligned}
		\end{equation*}
				for $n\ge n_0$, which proves (ii). Particularly, by (ii), we have
                \begin{equation*}
                    m^\mu (a_2)<\frac{a_2^2}{a_1^2}m^\mu (a_1)<m^\mu (a_1)<0,
                \end{equation*}
                which means $m^\mu (a)$ is decreasing in $a>0$.

                \vskip 0.1in

                (iii)
                For any $\epsilon > 0$ and $k\geq1$, we can find $\phi_{\epsilon}, \psi^j_{\epsilon} \in C_0^{\infty}(\R^N)$  $(j=1,\cdots,k)$ such that
                \begin{equation*}
                	\phi_{\epsilon}\in \mathcal{S}_b,  \quad \psi^j_{\epsilon}\in \mathcal{S}_{c_j},
                \end{equation*}
                \begin{equation*}
                	 J_{\mu}(  \phi_{\epsilon})\leq m^\mu (b)+\frac{\epsilon}{k+1},\quad  J_{ \infty}(  \psi^j_{\epsilon})\leq m^{0}(c_j)+\frac{\epsilon}{k+1}.
                \end{equation*}
                Let $u_{\epsilon,n}(x):=\phi_{\epsilon}(x)+\sum_{j=1}^{k}\psi^j_{\epsilon}(x-y_n^j)$ with $|y_n^j|\rightarrow+\infty,|y_n^i-y_n^j|\rightarrow+\infty$ for any $i\neq j$ as $n\rightarrow\infty$. Since $\phi_{\epsilon}, \psi^j_{\epsilon}$ have compact support and there exists $n_0>0$ such that
                $${\rm supp} \,\psi^i_{\epsilon}(\cdotp-y^j_n)\cap {\rm supp}\, \psi^i_{\epsilon}(\cdotp-y^i_n)=\emptyset\qquad\forall\, 1\le i\ne j\le k,n\ge n_0.$$
                 Thus for $n
\ge n_0$ one has 
                \begin{equation}\label{e2.17}
                	u_{\epsilon,n}\in \mathcal{S}_{a},\quad m^\mu (a)\leq  J_{\mu}(u_{\epsilon,n})=  J_{\mu}(  \phi_{\epsilon}(x))+\sum_{j=1}^{k}J_{\mu}(  \psi^j_{\epsilon}(x-y_n^j)).
                \end{equation}
                Thus we have as $n\rightarrow \infty$,
                \begin{equation}\label{e2.18}
                	 J_{\mu}(  \psi^j_{\epsilon}(x-y_n^j))\rightarrow  J_{\infty}(  \psi^j_{\epsilon}(x)).
                \end{equation}
                From \eqref{e2.17}-\eqref{e2.18}, it follows that
                \begin{equation*}
                	\begin{aligned}
                		m^\mu (a)&\leq\lim_{n\rightarrow\infty}  \Big( J_{\mu}(  \phi_{\epsilon}(x))+\sum_{j=1}^{k} J_{\mu}(  \psi^j_{\epsilon}(x-y_n^j))\Big)\\&= J_{\mu}(  \phi_{\epsilon}(x))+\sum_{j=1}^{k}J_{\infty }(  \psi^j_{\epsilon}(x))\\&\leq m^\mu(b)+\sum_{j=1}^{k}m^{0}(c_j)+\epsilon.
                	\end{aligned}
                \end{equation*}
                By the arbitraries of $\epsilon$, the result holds. 
			\end{proof}

       Set the following notations:
        $$J_{ \mu,\lambda}(u):= J_{\mu}(u) +\frac{\lambda}{2}\int_{\R^N}u^2 dx,\quad J_{0,\lambda}(u):=J_{0}(u)+\frac{\lambda}{2}\int_{\R^N}u^2 dx,\quad\forall~u\in H^1(\R^N).$$
        
      To show the behavior of the minimizing sequence, we prove the following splitting lemma.
        \vskip0.05in
       
        \begin{lemma}[Splitting Lemma]\label{splitting}
        Assume that $N\geq3,(f_1)-(f_3)$   hold,    $\mu\in[0,\bar{\mu})$  and    $h\in L^\infty(\R^N)$ with $\|h\|_\infty\le 1$. For fixed   $\lambda>0$, let $\{u_n\}\subset H^1(\mathbb{R}^N)$ be a bounded $(PS)_c$ sequence for $J_{\mu,\lambda}$   such that $$  u_n\rightharpoonup u  \quad \text{in}\quad    H^1(\mathbb{R}^N).$$ Then either $u_n\to u$ strongly in $H^1(\R^N)$ or there exists a number $k\in\mathbb{N}$, $k$ non-trivial solutions $w^1,\cdots,w^k\in H^1(\mathbb{R}^N)$ to the limit equation (\ref{q_infty}) 
        	and $k$ sequences $\{y_n^j\}_n\subset \mathbb{R}^N,1\leq j\leq k$ such that $|y_n^j|\rightarrow\infty,|y_n^i-y_n^j|\rightarrow +\infty$ for $i\neq j$ and
        	\begin{equation}\label{e2.20}
        		u_n=u+\sum_{j=1}^kw^j(\cdot-y_n^j )+o(1)~\text{strongly~in}~H^1(\mathbb{R}^N).
        	\end{equation}
        	Furthermore, one has
        	\begin{equation}\label{e2.21}
        		\|u_n\|_2^2=\|u\|_2^2+\sum_{j=1}^k\|w^j\|_2^2+o(1)
        	\end{equation}
        	and
        	\begin{equation}\label{e2.22} 
        		J_{\mu,\lambda}(u_n)= J_{\mu,\lambda}(u)+\sum_{j=1}^k J_{0,\lambda}(w^j)+o(1)
        	\end{equation}
        \end{lemma}
        The proof of the Splitting Lemma \ref{splitting} follows the outlines of the one of Lemma $3.1$ of \cite{LLT=JDE}, the main difference being the presence of a non-constant function $h\in L^\infty(\R^N)$ with $||h||_\infty\le 1$.
        \begin{proof}
        	Since $u_n \rightharpoonup u$  in $H^1(\mathbb{R}^N)$, thus for any $\psi\in C_0^{\infty}(\R^N)$,  
        	\begin{equation*}
        		0=\left\langle  J_{\mu,\lambda}(u_n) ,\psi \right\rangle+o(1)=\left\langle  J_{\mu,\lambda}(u) ,\psi\right\rangle,
        	\end{equation*}
        	this implies that $u$ is a solution of problem \eqref{q}. 
        	 
        	\vskip 0.1in
        	
         	Setting $\psi^1_n:=u_n-u$, we have $\psi^1_n\rightharpoonup 0$ weakly in $H^1(\R^N)$, $\psi^1_n\to 0$ a.e. in $\R^N$ and strongly in $L^p_{loc}(\R^N)$, hence by direct calculation, we have as $  n\to\infty $,
        	\begin{equation}
        		\label{psi-H1}\|\nabla\psi^1_n\|^2_2=\|\nabla u_n\|^2_2-\|\nabla u\|^2_2+o(1) ,
        	\end{equation}
        	\begin{equation}
        		\label{psi-L2}\|\psi^1_n\|^2_2=\|u_n\|^2_2-\|u\|^2_2+o(1),
        	\end{equation}
        	\begin{equation*}
        		\int_{\R^N}\frac{ h(x)|\psi^1_n|^2}{|x|^2}dx=	\int_{\R^N}\frac{ h(x)|u_n|^2}{|x|^2}dx-	\int_{\R^N}\frac{ h(x)|u|^2}{|x|^2}dx+o(1). 
        	\end{equation*}
        	Recall that $$f(t)\le c_1 t^{q-1}+c_2 t^{p-2},\qquad\forall\,t>0,$$
        for some constants $c_1,\,c_2>0$, due to $(f_1)-(f_3)$, which yields that 
        \begin{equation}
        \label{pt-conv}
        \int_{\R^N} f(u_n-\theta u)udx\to \int_{\R^N} f(u-\theta u)udx\qquad n\to\infty
        \end{equation}
        for any $\theta\in[0,1]$. In fact, due to the H\"{o}lder inequality, for any $\varepsilon>0$ there exists $R=R(\varepsilon)>0$ such that
        \begin{equation}\notag
        \begin{aligned}
        &\int_{\R^N\setminus B_R(0)}|f(u_n-\theta u)u-f(u-\theta u)u|dx\le c\int_{\R^N\setminus B_R(0)}(|u_n-\theta u|^{p-1}+|u_n-\theta u|^{q-1})|u| dx\\
        &\le c((\|u_n\|_q^{q-1}+\|u\|_q^{q-1})\|u\|_{L^q(\R^N\setminus B_R(0))}+(\|u_n\|_p^{p-1}+\|u\|_p^{p-1})\|u\|_{L^p(\R^N\setminus B_R(0))})\le \varepsilon/2, \qquad\forall\, n\in\N.
        \end{aligned}
        \end{equation}
        Moreover, using that $u_n\to u$ strongly in $L^r_{loc}(\R^N)$, for $r\in[2,2^*]$, due to the dominated convergence Theorem applied in $B_R(0)$ for any fixed $\theta\in[0,1]$, there exists $n_0=n_0(\varepsilon)>0$ such that
        $$\int_{B_R(0)}|f(u_n-\theta u)u-f(u-\theta u)u|dx\le\varepsilon/2\qquad\forall\, n\ge n_0,$$
        which concludes the proof of \eqref{pt-conv}. As a consequence, using that the integrals $\int_{\R^N} f(u_n-\theta u)ud\theta$ are bounded in $\theta\in[0,1]$, the Fubini Theorem and the dominated convergence Theorem give
        \begin{equation}\label{eq2.25}
        			\begin{aligned}
        	&	\int_{\R^N} F ( u_n)-F (\psi^1_n)dx=-\int_{\R^N}\int_0^1\frac{d}{d\theta}F (u_n-\theta u)d\theta dx =\int_{\R^N}\int_0^1f (u_n-\theta u)ud\theta dx\\&\rightarrow \int_{\R^N}\int_0^1f (u-\theta u)ud\theta dx =-\int_{\R^N}\int_0^1 \frac{d}{d\theta}F(u -\theta u)d\theta dx =	\int_{\R^N} F( u )dx.
        		\end{aligned}
        	\end{equation} 
        Thus above results yield that
        	\begin{equation}
        		\label{rel-nabla-I}
        	J'_{\mu,\lambda}(\psi^1_n)=J'_{\mu,\lambda}(u_n)-J'_{\mu,\lambda}(u)+o(1)\quad\text{in}\quad H^{-1}(\R^N)\quad\text{as}\quad  n\to\infty ,
        	\end{equation}
        		\color{black}\begin{equation}
        		\label{rel-nabla-I}
        	 J_{\mu,\lambda}(\psi^1_n)= J_{\mu,\lambda}(u_n)- J_{\mu,\lambda}(u)+o(1) \quad\text{as}\quad  n\to\infty.
        	\end{equation} 

   \color{black}Assume first that
   
   \begin{equation}
   \label{vanishing}
   \lim\limits _{n\rightarrow\infty}\sup\limits_{y\in\mathbb{R}^N}\int_{B_1(y)}|\psi^1_n|^2dx=0.
   \end{equation}
   Then by Lemma \ref{lemma1}, $\psi_1^n\to 0$ in $L^p(\R^N)$. Using that $\langle\nabla J_{\mu,\lambda}(\psi^1_n),\psi^1_n\rangle=o(1)$, we can see that
   $$0\le \frac{\bar{\mu}-\mu}{\bar{\mu}}||\nabla \psi^1_n||^2_2+\lambda||\psi^1_n||^2_2\le||\nabla \psi^1_n||^2_2-\mu\int_{\R^N}\frac{h(x)}{|x|^2}|\psi^1_n|^2(x)dx+\lambda||\psi^1_n||^2_2=\int_{\R^N}f(\psi^1_n)\psi^1_n dx=o(1),$$
   which yields that $\psi^1_n\to 0$ strongly in $H^1(\R^N)$ and we are done.\\ 
   
   Hence let us assume that \eqref{vanishing} does not hold. In this case, there exist $\delta_1>0$ and a sequence  $\{y_n^1\}\subset \mathbb{R}^N$ such that
        \begin{equation} 
         \int_{B_1(y_n^1)}|\psi^1_n|^2dx\geq \delta_1.
        \end{equation}
      We remark that $\{y_n^1\}$ is unbounded. Indeed, if the contrary were true, there would exist $R > 0$ such that 
       \begin{equation*} 
      \int_{B_R(0)}|\psi^1_n|^2dx\geq 	\int_{B_1(y_n^1)}|\psi^1_n|^2dx\geq \delta_1,
      \end{equation*}
      which contradicts the fact that $\psi^1_n\to 0$ in $L^2_{loc}(\R^N)$. Define $w_n^1 := u_n(\cdot+y_n^1)$, one can see
        \begin{equation} 
      	\int_{B_1(0)}|w_n^1|^2dx=	\int_{B_1(y_n^1)}|u_n|^2dx=  \int_{B_1(y_n^1)}(|u |^2+|\psi^1_n|^2)dx+o(1)
        \geq \frac{	\delta_1 }{2}.
      \end{equation}
        	As a consequence, $w^1_n \rightharpoonup w^1	\ne 0$ weakly in $H^1(\R^N)$, $w^1_n\to w^1			 $ a.e. in $\R^N$ and strongly in $L^p_{loc}(\R^N)$ for any $p\in[2,6).$ For any $\varphi\in C_0^{\infty}(\R^N)$ with  supp $\varphi\subset B_R(0)$ for some $R>0$, applying the H\"{o}lder inequality, we have as $n\rightarrow \infty$,
        	\begin{equation*}
        		\begin{aligned}
        		\int_{\R^N}\frac{h(x) u_n \varphi(\cdot-y_n^1)}{|x|^2}dx&\leq \Big( 	\int_{\R^N}\frac{ |u_n|^2}{|x|^2}dx\Big)^{\frac{1}{2}} \Big( 	\int_{\R^N}\frac{ |\varphi(\cdot-y_n^1)|^2}{|x|^2}dx\Big)^{\frac{1}{2}}\\&\leq 
        		C   \Big( 	\int_{\text{supp}\varphi}\frac{ |\varphi(x)|^2}{|x+y_n^1|^2}dx\Big)^{\frac{1}{2}}\leq C\frac{ ||\varphi||_2}{\Big| |y_n^1|-R\Big|}\rightarrow 0
        	\end{aligned}
        	\end{equation*}
        	Therefore, using that $J'_{\mu,\lambda}(u_n)=o(1)$ in $H^{-1}(\R^N)$, we obtain
        	\begin{equation*}
        		\begin{aligned}
        		o(1)&= \langle J_{\mu,\lambda}(u_n),\varphi(\cdot-y_n^1)\rangle \\&=	\int_{\R^N} \nabla u_n\nabla \varphi(\cdot-y_n^1) dx -\mu 	\int_{\R^N}\frac{ h(x)u_n \varphi(\cdot-y_n^1)}{|x|^2}dx+\lambda \int_{\R^N} u_n  \varphi(\cdot-y_n^1) dx -	\int_{\R^N} f(u_n) \varphi(\cdot-y_n^1)dx 
        		\\&=
        			\int_{\R^N} \nabla w^1_n\nabla \varphi(x) dx  +\lambda \int_{\R^N} w_n^1  \varphi(x) dx
        	  -	\int_{\R^N} f(w^1_n) \varphi(x)dx \\&
        	 = \langle J_{\infty,\lambda}(w^1			 ),\varphi\rangle+o(1),
        		\end{aligned}
        	\end{equation*}
        	that is, $w^1$ is a solution to
         \begin{equation}\label{lim-R}
        	-\Delta u +\lambda u =f(u)\quad\text{in}\quad\mathbb{R}^{N}.
        \end{equation}
        	Setting $ \psi^2_n =u_n-u-w^1(\cdotp-y^1_n)$, then
        	\begin{equation}\label{e2.35}
        		|| \psi^2_n||_2^2=||u_n||_2^2-||u ||_2^2-||w^1			 ||_2^2.
        	\end{equation}
        	We claim that
        	\begin{equation}\label{e2.36}
        		\begin{aligned}
        		\int_{\R^N}\frac{h(x) | \psi^2_n|^2 }{|x|^2}dx &=\int_{\R^N}\frac{h(x) | u_n-u|^2 }{|x|^2}dx-\int_{\R^N}\frac{ h(x)| w^1(\cdotp-y^1_n)|^2 }{|x|^2}dx+o(1) \\&
        		=
        		\int_{\R^N}\frac{h(x) | u_n|^2 }{|x|^2}dx-	\int_{\R^N}\frac{h(x) |  u|^2 }{|x|^2}dx+o(1).
        			\end{aligned}
        	\end{equation}
        	In fact one can find a sequence $\{\varphi_m\}\subset C_0^{\infty}(\R^N)$
          such that $\varphi_m\rightarrow w^1$ in $H^1(\R^N)$ as $m\rightarrow \infty$. Let $R_m > 0$
        	satisfy supp $\varphi_m\subset B_{R_m} (0)$. For any $m\in \mathbb{N}$, we find $n = n(m)$ such that $|y_n| -R_m \geq  m$ and $\{n(m)\}$ is
        	an increasing sequence. Then we get as $n\rightarrow \infty$,
        	\begin{equation*}
        		\begin{aligned}
        			 \int_{\R^N}\frac{  h(x)|\varphi_m(\cdot-y_n^1)|^2}{|x|^2}dx&\leq  \int_{B_{R_m} (0)}\frac{ |\varphi_m(x)|^2}{|x+y_n^1|^2}dx \\&\leq \frac{1}{\Big( |y_n^1|-R_m\Big)^2}\int_{B_{R_m} (0)} |\varphi_m(x)|^2 dx\\&\leq\frac{1}{m^2}||
        			 \varphi_m(x)||_2^2\rightarrow 0, 
        		\end{aligned}
        	\end{equation*}
        	which deduces that
        		\begin{equation}\label{eq2.39}
        		\int_{\R^N}\frac{ h(x) |w^1			 (\cdot-y_n^1)|^2}{|x|^2}dx \rightarrow 0, 
        	\end{equation}
        	Since $u_n\rightharpoonup u$ weakly in $H^1(\R^N)$, the claim holds. Observe that  $\psi^2_n(\cdot+y_n^1)\rightharpoonup w^1			 $ a.e. in  $\R^N$,
        then by Vitali's Theorem and \eqref{eq2.25}, we get
        	\begin{equation}\label{e2.37}
        		\begin{aligned}
         	\int_{\R^N} F(\psi^2_n )dx=&	\int_{\R^N} F(\psi^2_n (\cdot+y_n^1))dx=	\int_{\R^N} F(u_n-u)dx-\int_{\R^N} F(w^1			 )dx+o(1)\\&
         	=	\int_{\R^N} F(u_n )dx-	\int_{\R^N} F(u  )dx-\int_{\R^N} F_R(w^1			 )dx+o(1).
         	\end{aligned}
        \end{equation}
        
        	It follows from \eqref{e2.35}-\eqref{e2.37}, one can conclude that
        	\begin{equation*}
        		J_{\mu,\lambda}(\psi^2_n)=	J_{\mu,\lambda}(u_n)-	J_{\mu,\lambda}(u)-J_{0,\lambda}(w^1)+o(1).
        	\end{equation*}
        	 Then we consider 
        	\begin{equation} 
        		\delta_2:= \lim _{n\rightarrow\infty}\sup_{y\in\mathbb{R}^N}\int_{B_1(y)}|\psi^2_n|^2dx.
        	\end{equation}
            
        Iterating the above procedure, for any integer $\ell\geq1$, there exist $\ell$ non-trivial solutions $w^1,\cdots,w^\ell\in H^1(\mathbb{R}^N)$ to \eqref{lim-R} and $\ell$ sequences $\{y_n^j\}_n\subset \mathbb{R}^N$ satisfying $|y_n^j|\rightarrow\infty,|y_n^i-y_n^j|\rightarrow +\infty$ for $i\neq j$, $1\le j\le \ell$, such that, setting $\psi^{\ell+1}_j:= \psi^2_n =u_n-u-\sum_{j=1}^\ell w^j(\cdotp-y^j_n)$, we have
        	\begin{equation}\label{lim-norms}
        		\begin{aligned}
        			\|\nabla\psi^{k+1}_n\|^2_2&=\|\nabla\psi^{k}_n\|^2_2-\|\nabla w^k\|^2_2+o(1)=\|\nabla u_n\|^2_2-\|\nabla u\|^2_2-\sum_{j=1}^{k}\|\nabla w^{j}\|^2_2+o(1),\\
        			\|\psi^{k+1}_n\|^2_2&=\|\psi^{k}_n\|^2_2-\|w^k\|^2_2+o(1)=\|u_n\|^2_2-\|u\|^2_2-\sum_{j=1}^{k}\|w^{j}\|^2_2+o(1),\\
        					\int_{\R^N} F(\psi^{k+1}_n )dx&=\int_{\R^N} F(\psi^{k}_n )dx-\int_{\R^N} F(w^k )dx+o(1)\\&
        					=	\int_{\R^N} F(u_n )dx-	\int_{\R^N} F(u  )dx-\sum_{j=1}^{k}\int_{\R^N} F(w^j)dx+o(1).
        		\end{aligned}
        	\end{equation}
            as $n\to\infty$. As a consequence, the following results hold:
         \begin{itemize}
         	\item $J'_{\mu,\lambda}(u)=0$ and $J'_{0,\lambda}(w^j)=0$, where $w^j\neq0$ for $j=1,\cdots,k$,
         	\item  $\psi_n^{k+1}=u_n-u-\sum_{j=1}^{k}w^j(\cdot-y_n^j)$,
         	\item $\|\psi^{k+1}_n\|^2_2=\|u_n\|^2_2-\|u\|^2_2-\sum_{j=1}^{k}\|w^{j}\|^2_2+o(1),$
         	\item  $J_{\mu,\lambda}(u_n)=J_{\mu,\lambda}(u)+J_{\mu,\lambda}(\psi_n^{k+1})+\sum_{j=1}^kJ_{0,\lambda}(w^j)+o(1)$.
         \end{itemize}
    
         \vskip0.05in
         	To conclude the lemma, we need to prove 
       that the iterations must stop after finite steps. Since $w^j$ is a non-trivial solution to \eqref{lim-R},
        using \eqref{f'}, it holds
       \begin{equation} \label{q2.44}
       	 \int_{\R^N}|\nabla w^j|^2dx+\lambda \int_{\R^N}|  w^j|^2dx =\int_{\R^N}f(w^j)w^jdx\leq C_1 \int_{\R^N} |w^j|^qdx + C_2\int_{\R^N} |w^j|^pdx.
       \end{equation}
      Combining with Sobolev inequality
      \begin{equation*}
       	 || w||_p\leq c || w||_{H^1}\qquad\forall w\in H^{1}(\R^N),
       \end{equation*}
      one gets that 
       \begin{equation*}
       || w^j||_{H^1}^2\le C_3|| w^j||_{H^1}^q+C_4|| w^j||_{H^1}^p
       \end{equation*}
     for some constants $C_3,\,C_4>0$. Since $p,\,q>2$, we deduce that there is a constant $\hat{c}>0$ independent of $w^j$ such that $||w^j||_{H^1}\geq \hat{c} $. Since $\{u_n\}$ is
    bounded in $H^1(\R^N )$, the iterations must stop after finite steps, namely there exists $k\ge 1$ such that $\psi_n^{j+1}=0 $ for any $j\ge k+1$, which concludes the proof. 
    
        \end{proof}

		\color{black}	\begin{lemma} \label{lem2.6}
					Assume that  $N\geq3,(f_1)-(f_3)$  hold, $\mu\in(0,\bar{\mu})$  and    $h\in L^\infty(\R^N)$ with $\|h\|_\infty\le 1$. Let $\{u_n\}\subset \mathcal{S}_a$ be a minimizing sequence with respect to
				$m^\mu(a)$. Then, up to a subsequence, there exist  $a_0>0$ and  $u\in \mathcal{S}_a $  such that for any $a>a_0$, $u_n\rightarrow u$ in $H^1(\R^N)$ and $(u,\lambda)$ is a solution of \eqref{q}, where  $\lambda >0$ is the associated Lagrange multiplier. Particularly, $u(x)>0$ for any $x\in\R^N$.
			\end{lemma}
			\begin{proof}
			 Let $\{u_n\}\subset \mathcal{S}_a$ be a minimizing sequence for $m^\mu(a)$.  Since $J_{\mu}$ is coercive in $\mathcal{S}_a$   uniformly in $ a$\color{black}, $\{u_n\}$ is bounded in $H^1(\mathbb{R}^N)$  uniformly in $a$. \color{black}Note that for any $u\in H^1(\mathbb{R}^N),$ $J_{\mu}(u)=J_{\mu}(|u|)$, then one can choose $\{u_n\}$ being a non-negative sequence. From Lemma \ref{lemma2.4}, there exists a non-negative $(PS)$ sequence $\{v_n\}$ for $J_{\mu}$ constraint on $\mathcal{S}_a$ with $||v_n-u_n||_{H^1}\rightarrow0$. Thus one has as $n\rightarrow\infty$,
			\begin{equation}
				J_{\mu}(v_n)\rightarrow m^\mu(a),~ (J_{\mu}|_{\mathcal{S}_a})'(v_n)\rightarrow0\quad\text{in}\, H^{-1}(\R^N).
			\end{equation}
			Then the Lagrange multiplier 
			\begin{equation}\label{eq5.4}
				\lambda_n:= -\frac{J'_{\mu}(v_n)v_n}{a^2}
			\end{equation}
			is  also bounded since $\{v_n\}$ is bounded in $H^1(\mathbb{R}^N)$  uniformly in $a $\color{black}. Up to a subsequence, there exist 
            $v\in H^1(\mathbb{R}^N)$ with $v\geq0$ and $\lambda\in \mathbb{R}$ such that 
            $$v_n\rightharpoonup v\quad\text{ in}\quad  H^1(\mathbb{R}^N)\quad\text{and}\quad \lambda_n\rightarrow \lambda\quad\text{ in}\quad  \mathbb{R}.$$
     Now we show $\lambda_n>0$ for $n$ large. On one hand, we note  that
       \begin{equation*}
       		\int_{\R^N}| \nabla v_n|^2dx-	\mu\int_{\R^N}\frac{h(x)}{|x|^2}v_n^2dx+\lambda_n	\int_{\R^N} v_n^2 dx=	\int_{\R^N} f(v_n)v_ndx+o(1) 
       \end{equation*}
       and 
         \begin{equation*}
        m^\mu(a)=\frac{1}{2}	\int_{\R^N}| \nabla v_n|^2dx-\frac{\mu}{2}	\int_{\R^N}\frac{h(x)}{|x|^2}v_n^2dx-	\int_{\R^N} F(v_n) dx+o(1).
        \end{equation*}
        Thus it yields  as $n\rightarrow +\infty$, 
        \begin{equation}\label{lambda}
        	\lambda_na^2=-2m(a)+\int_{\R^N} f(v_n)v_ndx -2	\int_{\R^N} F(v_n) dx+o(1)\geq -2m(a)+(q-2)\int_{\R^N} F(v_n) dx+o(1)
        \end{equation}
since $f(t)t\geq qF(t)$ due to $(f_2)$. On the other hand,  using $(f_1)-(f_3)$, we can see that $f(t)>0$ for $t>0$ and $f(t)<0$ for $t<0$, which yields that $F(0)=0<F(t)$ for any $t\ne 0$, hence $\int_{\R^N} F(v_n) dx>0$. 
\color{black}Since $m^\mu(a)<0$, one gets $\lambda_n a^2\ge-2m^\mu(a)>0$ for any $n$ sufficiently large and so that $\lambda\ge-\frac{2 m^\mu(a)}{a^2}>0$.
\vskip 0.1in

             Applying Lemma \ref{splitting}, either $v_n\to v$ strongly in $H^1(\R^N)$ and we are done, or there exists an integer $k\ge 1$, $k$ non-trivial solutions $w^1,\cdots,w^k\in H^1(\mathbb{R}^N)$ to the limit equation
			 \begin{equation} 
			 	-\Delta u+\lambda u =f(u)\quad\text{in}\quad\R^N
			 \end{equation}
			 and $k$ sequences $\{y_n^j\}\subset \mathbb{R}^N,1\leq j\leq k$ such that $|y_n^j|\rightarrow\infty,|y_n^i-y_n^j|\rightarrow +\infty$ for $i\neq j$ and
			 \begin{equation} \label{eq2.44}
			 	v_n=v+\sum_{j=1}^kw^j(\cdot-y_n^j )+o(1)~\text{strongly~in}~H^1(\mathbb{R}^N).
			 \end{equation}
			 Furthermore, one has
			 \begin{equation} \label{e2.45}
			 	\|v_n\|_2^2=\|v\|_2^2+\sum_{j=1}^k\|w^j\|_2^2+o(1)
			 \end{equation}
			 and
			 \begin{equation} 
             \label{e2.46}
			 	J_{\mu,\lambda}(v_n)= J_{\mu,\lambda}(v)+\sum_{j=1}^k J_{0,\lambda}(w^j)+o(1).
			 \end{equation}
		 Set $b:=||v||_2$ and $ c_j:=||w^j||_2\ge 0$, for $1\le j\le k$. Then $b\in[0,a)$, since $k\ge 1$, 
		 and $a^2=b^2+\sum_{j=1}^{k}c^2_j$.
 Note that for each $j=1,\cdots,k$, 
                 \begin{equation*}
                    -\Delta w^j+\lambda w^j= f(w^j) .
                \end{equation*}
        \eqref{e2.45} and \eqref{e2.46} yields that
                \begin{equation*}
                	J_{\mu }(v_n)= J_{\mu }(v)+\sum_{j=1}^k J_{0}(w^j)+o(1).
                \end{equation*}
                As a consequence,  setting $m^\mu(0)=0$, \color{black}we have
                  \begin{equation}\label{ep1}
                \begin{aligned}
                   m^\mu(a)&=\lim_{n\rightarrow \infty}J_{\mu }(v_n)\\&=
                    \lim_{n\rightarrow \infty}\Big(J_{\mu }(v ) +\sum_{j=1}^k J_{0}(w^j)+o(1)\Big) 
                  \\&                    \geq m^\mu(b)+\sum_{j=1}^km^{0}(c^j).
                \end{aligned}
                  \end{equation}
                      From Lemma \ref{lem2.3}-(iii), we have $  m^\mu (a)\leq m^\mu (b)+\sum_{j=1}^{k} m^{0} (c_j)$, thus
                      \begin{equation}\label{rel-min-energy}
                      m^\mu(a)= m^\mu(b)+\sum_{j=1}^{k} m^{0} (c_j),
                      \end{equation}
                     that implies the  inequality in \eqref{ep1} is actually an identity, from where one can conclude
                     \begin{equation} \label{eq2.47}
                      J_{\mu }(v )=m^\mu (b),~J_{0}(w^j)=m^{0}(c_j) ~\text{for~ each}~ j=1,\cdots,k.
                      \end{equation}
                     
                 \textbf{(a)}  If $b=0$ and $k\geq2$, then $v =0,a^2= \sum_{j=1}^{k}c^2_j$ and 
            \begin{equation*}  
           	v_n= \sum_{j=1}^kw^j(\cdot-y_n^j )+o(1)~\text{strongly~in}~H^1(\mathbb{R}^N).
           \end{equation*}
      Thus, using that $h\ge 0$ we have
      $$m^\mu (a)=\sum_{j=1}^{k}m^{0}(c_j)\geq \sum_{j=1}^{k}m^\mu (c_j)> \sum_{j=1}^{k}\frac{c_j^2}{a^2}m^\mu(a)=m^\mu(a),$$ which is a contradiction.  
      \vskip 0.1in 
           If $b=0$ and $k=1,$ then $v =0,a^2=  c^2_1$, $v_n=w^1(\cdot  -y_n^1)$ and $m^\mu(a)=m^{0}(a).$  By \cite[Theorem 1.1]{shibata}, there exists $a_0>0$ such that for any $a>a_0$, $m^{0}(a)$ is achieved by $v_a\in \mathcal{S}_a$ with $v_a>0$, thus, using that $h\ne 0$ and $h\ge 0$, we have
       \begin{equation*}
        m^\mu(a)\leq J_{\mu}(v_a)<J_{0}(v_a)=m^{0}(a),
       \end{equation*}
       which is impossible. We stress that here it is the only step in which we use that $\mu\ne 0$ and $h\ne 0$. 
               \vskip 0.1in

  \textbf{(b)}  If  $b\in(0,a)$ and $k\geq1$,  then $v_n=v+\sum\limits_{j=1}^kw^j(\cdot-y_n^j )+o(1)$ and $a^2= b^2+\sum\limits_{j=1}^{k}c^2_j$. As a consequence, by \eqref{rel-min-energy} and the fact that $h\ge 0$, one has
  \begin{equation*}
       m^\mu(a)= m^\mu(b)+\sum\limits_{j=1}^{k} m^{0} (c_j)\geq m^\mu(a)= m^\mu(b)+\sum_{j=1}^{k} m ^\mu(c_j)>\frac{b^2+\sum\limits_{j=1}^kc_j^2}{a^2}m^\mu(a)=m^\mu(a),
  \end{equation*}
				which is also impossible.    
				\vskip0.1in
			\color{black}	Combining \textbf{(a)} with  \textbf{(b)}, one can conclude that $b=a$ and $v_n\rightarrow v $ strongly in $H^1(\mathbb{R}^N)$.
				Now we prove $v >0$ in $\R^N$. It follows from \eqref{eq2.44}-\eqref{e2.46} that 
				\begin{center}
					$J_{\mu }(v )=m^\mu(a)$ and $v_n\rightarrow v $ in $H^1(\R^N)$.
				\end{center}
				Then by the Harnack inequality \cite{Tru}, we know $v >0$ in $\R^N$. In fact, if there exists $x_0\in\R^N$ such that
$v (x_0) = 0.$ Since $v\geq 0$,  there exists $x_1 \in \R^N$ such that $v (x_1) >
 0$. Have this in mind,
fix $L > 0$ large enough such that $x_0, x_1 \in B_L(0)$. By \cite[Theorem 8.20]{Tru}, 
there exists $C > 0$ such that
\begin{equation*}
    \sup_{y\in B_L(0)}v (y)\leq C \inf_{y\in B_L(0)}v (y)
\end{equation*}
which is impossible. Therefore, $v(x)>0$ for any $x\in\R^N$.  
				
				\vskip0.05in 
			\color{black} Set $u :=v $, this lemma is proved.
			\end{proof}
	\begin{remark}
	Theorem \ref{thm1.1} follows from Lemma \ref{lem2.2} and \ref{lem2.6}.
\end{remark}

 We emphasize that the mass-subcritical growth of $f$ is crucial for the existence of a minimizer. Indeed, as the following remark shows, this is generally not possible if $f$ has mass-supercritical growth. 

 \begin{remark}\label{rem2.2}
 If $f(t)=|t|^{p-2}t$ with $p\in (2+\frac{4}{N}, 2^*)$ then the following is true:
 \begin{itemize}
     \item The functional $J_\mu$ is unbounded from below on $\mathcal{S}_a$ for any $\mu\in[0,\bar{\mu})$ and $a>0$. In fact, taking $u\in \mathcal{S}_a$ and recalling the scaling $u^t(\cdotp):=t^{N/2}u(t\cdotp)$, for $t>0$ we have $||u^t||_2=||u||_2$ and
     $$J_\mu(u^t)\le\frac{t^2}{2}||\nabla u||^2-c_p t^{\frac{N(p-2)}{2}}||u||^p_p\to-\infty, \quad \text{as}\quad t\to\infty,$$
     for any $\mu\in[0,\bar{\mu})$.
     \item If $V\in L^s(\R^N)$ for some $s\in[\frac{N}{2},\infty)$, then the functional $E$ defined in \eqref{def-E} is unbounded from below on $\mathcal{S}_a$ for any $a>0$. In fact, the H\"{o}lder inequality and the Sobolev embedding $D^{1,2}(\R^N)\subset L^{\frac{2s}{s-2}}(\R^N)$ yields that
     $$\left|\int_{\R^N}V(x)u^2(x)dx\right|\le ||V||_s || u||_{\frac{2s}{s-2}}^2\le c_s||V||_s ||\nabla u||^2_2,$$
     so that
     $$E(u^t)\le \frac{1}{2}(1+c_s||V||_s)||\nabla u||^2_2-c_p t^{\frac{N(p-2)}{2}}||u||^p_p\to-\infty, \qquad \text{as}\quad  t\to\infty.$$
     \item Similarly, if $V\in L^\infty(\R^N)$, we have
     $$E(u^t)\le \frac{1}{2}||\nabla u||^2_2+\frac{||V||_\infty a^2}{2}-c_p t^{\frac{N(p-2)}{2}}||u||^p_p\to-\infty, \qquad  \text{as}\quad  t\to\infty.$$
 \end{itemize}
 As a consequence, no minimizer of the aforementioned functionals can be found on $\mathcal{S}_a$.
 \end{remark}

 \section{Proof of Theorem \ref{thm-rad}}\label{sect3}
 \begin{proof}[\bf Proof of existence of Theorem \ref{thm-rad}]
Lemmas \ref{lem2.2} and \ref{lemma2.4} hold true in the radial setting too, that is replacing $H^1(\R^N)$ by $H^1_r(\R^N)$ everywhere. In particular, we have $m_r^\mu(a)\in(-\infty,0)$. As a consequence, given $a>0$ and $\mu\in[0,\bar{\mu})$, we can construct a bounded Palais-Smale sequence $\{u_n\}\subset \mathcal{S}_{a,r}$ of $J_\mu$ such that $J_\mu(u_n)\to m^0_r(a)$. In other words, we have
 $$J_\mu(u_n)\to m^\mu_r(a),\quad J'_\mu(u_n)+\lambda_nu_n=o(1)\quad\text{in}\quad H^{-1}(\R^N),$$
for some bounded sequence $\{\lambda_n\}\subset\R$. As a consequence, there exist $u^\mu\in H^1_r(\R^N)$ and $\lambda^\mu\in\R$ such that, up to a subsequence $u_n\rightharpoonup u^\mu$ and $\lambda_n\to\lambda^\mu\ge-\frac{2 m^\mu(a)}{a^2}>0$. In particular, we have
$$-\Delta u^\mu-\mu\frac{h(x)}{|x|^2}u^\mu+\lambda^\mu u^\mu=f(u^\mu)\qquad\text{in}\quad\R^N,\quad||u^\mu||_2\le a.$$
It remains to prove that $u_n\to u^\mu$ strongly in $H^1(\R^N)$, which proves that $||u^\mu||_2=a$ and $J_\mu(u^\mu)=m^\mu_r(a)$.\\

In order to do so, we recall that, thanks to the hardy inequality we have
 $$\int_{\R^N}|\nabla u|^2dx-\mu\int_{\R^N}\frac{h(x)u^2}{|x|^2}dx+\lambda^\mu\int_{\R^N}u^2 dx\ge\frac{\bar{\mu}-\mu}{\bar{\mu}}||\nabla u||^2_2+\lambda_\mu||u||_2^2\qquad\forall\,u\in H^1(\R^N),$$
which yields that the operator
\begin{equation}
\label{def-L-mu}
L_\mu:=-\Delta-\frac{\mu h(x)}{|x|^2}+\lambda^\mu:H^1(\R^N)\to H^{-1}(\R^N)
\end{equation}
is invertible with bounded inverse for any $\mu\in(0,\bar{\mu})$. Moreover, due to the compactness of the embedding $H^1_r(\R^n)\subset L^s(\R^N)$, for $s\in(2, 2^*)$, we have
$$u_n=L_\mu^{-1}(f(u_n)+o(1))\to L_\mu^{-1}(f(u^\mu))\quad\text{in}\quad H^1(\R^N).$$
Combining this with the weak convergence $u_n\rightharpoonup u^\mu$ in $H^1(\R^N)$,   we conclude the proof.
\end{proof}

To summarize the proof of Theorem \ref{thm-rad}, it only remains to prove the properties of $u^{\mu}$ for the special case $h=1$, that will be concluded by recalling the following result, which was proved in \cite{LLT=JDE}.

\begin{remark}\label{rem2.1}
	\label{lem3.1}  
	If $h=1$  and $\mu\in(0,\bar{\mu})$ then for $a>a_0$, let $(u^{\mu} ,\lambda ^{\mu} )\in H^1(\R^N)\times \R^+$ be a  solution to \eqref{q}, then by \cite[Proposition 3.3, Proposition 3.5]{LLT=JDE}, we can see that  $u^\mu(x)\in C^2(\R^N \setminus \{0\})$ and
	there are positive constants $\delta$ and $R$ such that for $|\alpha| \leq 2$,
	\begin{equation*}
		|D^{\alpha}u^{\mu}(x)|\leq C\,\text{exp}(-\delta\, |x|),\quad \forall \, |x|\geq R
	\end{equation*}and 
	\begin{equation*}
		u^{\mu} (x)\leq C |x|^{-(\sqrt{\bar{\mu}}-\sqrt{\bar{\mu}-\mu} )},\quad\forall \,x\in B_{\rho}(0)\setminus \{0\},
	\end{equation*}
	where $\rho>0$ is suitably small.
\end{remark} 
                  \color{black}      
						\section{Proof of Theorem \ref{th1.2}}\label{se2}
						Now we study the asymptotic behavior with respect to $\mu\rightarrow 0^+$. We begin with the following Lemma.
                        \begin{lemma}\label{lemma-rad-rearr}
                        In the above notations, we have $m^0_r(a)=m^0(a)$, for any $a>0$.
                        \end{lemma}
\begin{proof}
On the one hand, using that $\mathcal{S}_{a,r}\subset\mathcal{S}_a$, we have $m^0_r(a)\ge m^0(a)$. On the other hand, given a function $u\in H^1(\R^N)$, the radial rearrengment $u^*\in H^1_r(\R^N)$ constructed by Kawohl \cite{K} fulfills
$$\int_{\R^N}|\nabla u^*|^2dx\le \int_{\R^N}|\nabla u^*|^2dx,\quad \int_{\R^N}f(u^*)dx=\int_{\R^N}f(u)dx,\quad ||u^*||_2=||u||_2,$$
so that $J_0(u^*)\le J^0(u)$, which yields that $m^0_r(a)\le m^0(a)$.
\end{proof}

						\begin{lemma}\label{lemma2.9}
							Assume that  $(f_1)-(f_3)$   hold and $\mu\in(0,\bar{\mu})$.  Let $(u^{\mu} ,\lambda^{\mu} )\in H^1(\R^N)\times [-\frac{2 m_r^\mu(a)}{a^2},0)$ be the global minimizer to \eqref{q} and $m^{\mu}_r(a)$ be the least energy, then
                            $$m^{\mu}_r(a)\rightarrow m^0_r(a),\quad (u^{\mu},\lambda^{\mu})\rightarrow (u^{0},\lambda^{0})\quad\text{in}\quad H^1(\R^N)\times \R,$$ where $(u^{0},\lambda^{0})\in H^1(\R^N)\times [-\frac{2 m^0(a)}{a^2},0)$ is a radial solution of the limit equation \eqref{q_infty}.
						\end{lemma}
						\begin{proof}
							First we show that for given $\mu_0\in(0,\bar{\mu})$, $m^{\mu}(a)$ is uniformly continuous and non-increasing in $\mu\in[0,\mu_0]$.\\ 
                            
                            By the definition of $m^{\mu}_r(a)$ and the fact that $h\ge 0$, it is easy to see $m^{\mu}_r(a)$ is non-increasing in $[0,\mu_0]$. For the continuity, we need to prove
							\begin{equation}
								\label{continuity}\lim_{n\rightarrow+\infty} m_r^{\mu -\color{black}\frac{1}{n}}(a)\leq m^{\mu}_r(a) \leq 	\lim_{n\rightarrow+\infty} m_r^{\mu +\color{black}\frac{1}{n}}(a)
							\end{equation}
for $\mu\in(0,\mu_0)$.  For $\mu=0$, it is enough to prove the second inequality, while for $\mu=\mu_0$, it is enough to prove the first inequality.\\
                            
Let us prove the first inequality. For any $\varepsilon>0$ small enough, there exists $u\in\mathcal{S}_{a,r}$ such that $$J_\mu(u)<m^\mu_r(a)+\frac{\varepsilon}{2}<0.$$
Moreover, by relation \eqref{rel-coercivity}, for any $\mu\in[0,\mu_0]$ there exists a constant $\bar{C}>0$ such that, 
\begin{equation}
\label{equi-bd}
\text{if}~~ v\in\mathcal{S}_a~~ \text{satisfies}~~ J_\mu(v)<0,~~\text{then}~~||\nabla v||_2\le \bar{C}.  
\end{equation}
As a consequence
\begin{equation}
\begin{aligned}
m^{\mu-\frac{1}{n}}_r(a)&\le J_{\mu-\frac{1}{n}}(u)=J_\mu(u)+\frac{1}{2n}\int_{\R^N}\frac{h(x)u^2(x)}{|x|^2}dx\\
&<m^\mu_r(a)+\frac{\varepsilon}{2}+\frac{\bar{\mu}}{2n}||\nabla u||^2_2<m^\mu(a)+\varepsilon
\end{aligned}
\end{equation}
provided $n>n_0$, for some $n_0=n_0(\varepsilon,\mu)>0$. This concludes the proof of  the first inequality.\\

In order to prove the second inequality, we note that, for any $n>0$ there exists $u_n\in\mathcal{S}_a$ such that $$J_{\mu+\frac{1}{n}}(u_n)<m^{\mu+\frac{1}{n}}_r(a)+\frac{1}{n}<0.$$
As a consequence, using once again \eqref{equi-bd}, we can see that there exists $ \tilde{C}_\mu>0$ such that
\begin{equation}\notag
J_\mu(u_n)-\frac{\tilde{C}_\mu}{n}\le J_\mu(u_n)-\frac{1}{2n}\int_{\R^N}\frac{h(x)u_n^2(x)}{|x|^2}dx=J_{\mu+\frac{1}{n}}(u) 
<m^{\mu+\frac{1}{n}}_r(a)+\frac{1}{n},
\end{equation}
which yields that
\begin{equation}
    m^\mu_r(a)\le J_\mu(u_n)<m^{\mu+\frac{1}{n}}_r(a)+\frac{1+\tilde{C}_\mu}{n}.
\end{equation}
Taking the limit as $n\to\infty$, we have the statement.\\

In particular \eqref{continuity} shows that $m^\mu_r(a)\to m^0_r(a)$ as $\mu\to 0^+$.

							\vskip0.1in
                           
						 Let $\{\mu_n\}\subset(0,\infty)$ be a sequence such that $\mu_n\to 0^+$. For the sake of simplicity, we will drop the subscript $n$ and write $\mu:=\mu_n$. By \eqref{equi-bd}, we know $\{u^{\mu }\}\subset H^1_r(\R^N)$  is bounded in $H^1(\R^N)$ and by \eqref{eq5.4},  $\{\lambda^{\mu }\}$ is also bounded in $\R$. Then up to a subsequence,  there exist 
							$u^0\in H^1_r(\mathbb{R}^N)$ with $u^0\geq0$ and $\lambda^0\in [-\frac{2 m^0(a)}{a^2},0)$ such that 
							$$u^{\mu }\rightharpoonup u^0\quad\text{ in}\quad  H^1(\mathbb{R}^N)\quad\text{and}\quad \lambda^{\mu}\rightarrow \lambda^0\quad\text{ in}\quad  \mathbb{R}.$$ 
							We note that
                            $$L_0 u^\mu=f(u^\mu)+(\lambda^0-\lambda^\mu)u^\mu+\mu\frac{h(x)u^\mu}{|x|^2}\qquad\text{in}\,\R^N,$$
                            where $L_0$ is defined in \eqref{def-L-mu}, or equivalently
                            $$u^\mu=L_0^{-1}\left(f(u^\mu)+(\lambda^0-\lambda^\mu)u^\mu+\mu\frac{h(x)u^\mu}{|x|^2}\right).$$
						
                        We stress that $L_0$ is invertible thanks to the fact that $\lambda_0\ge-\frac{2 m^0(a)}{a^2}>0$. Using that $u^\mu\in H^1_r(\R^N)$ and the compactness of the embedding $H^1_r(\R^N)  \hookrightarrow L^s(\R^N)$, for $s\in(2, 2^*)$, we can see that  
                        $$u^\mu=L_0^{-1}\left(f(u^\mu)+(\lambda^0-\lambda^\mu)u^\mu+\mu\frac{h(x)u^\mu}{|x|^2}\right)\to L_0^{-1}(f(u^0)).$$
                        Since $u^\mu\rightharpoonup u^0$ in $H^1(\R^N)$, one can summarize
                        $$u^\mu\to u^0\quad\text{strongly}\quad 
                        \text{in}\quad H^1(\R^N)\quad$$
                        and
                        $$ -\Delta u^0+\lambda^0 u^0=f(u^0).$$
                        Since the sequence $\{\mu_n\}$ is arbitrary, the above results conclude that 
                        \begin{equation*}
                        	 (u^\mu,\lambda^\mu)\to(u^0,\lambda^0)\quad\text{in}\quad H^1(\R^N)\times\R\quad  \text{as}\quad \mu\to 0^+.
                        \end{equation*} 
                    \end{proof}
                    
                    Now it remains to prove that $J_0(u^0)=m^0(a)$. 
						\begin{remark}
							Theorem \ref{th1.2} follows from Lemmas \ref{lemma2.9} and \ref{lemma-rad-rearr}. In fact, due to the continuity of $J_\mu$ with respect to the parameter $\mu$ and to the variable $u$, we have
                            $J_\mu(u^\mu)\to J_0(u^0)$ as $\mu\to 0^+$. However, due to Lemma \ref{lemma2.9} and \ref{lemma-rad-rearr}, one has $$J_\mu(u^\mu)=m^\mu_r(a)\to m^0_r(a)=m^0(a)\quad\text{as}\quad \mu\to 0^+,$$
                            which yields that $J_0(u^0)=m^0(a)$. This means that $u^0$ is a minimizer of $J_0$ on $\mathcal{S}_a$.
						\end{remark}
              
				\qquad

							\noindent{\bf Declarations of interest}
							
							The authors have no interest to declare.

							~\\
							\color{black}\noindent{\bf Acknowledgment}
							
							X. Peng is supported by China Postdoctoral Science Foundation (2025M773085).  M. Rizzi is partly supported by the DFG project (62202684). The authors would also like to thank the referees for their valuable comments.


							
						\end{document}